\documentclass[11pt]{amsart}
\usepackage{texmac,tikz,pgfplots}
\usetikzlibrary{positioning}
  \pgfplotsset{compat=1.3}
\usepackage{algpseudocode}

\def\cO{\mathcal{O}}
\def\cK{\mathcal{K}}
\def\cL{\mathcal{L}}

\def\fp{\mathfrak{p}}

\def\FF{\mathbb{F}}
\def\ZZ{\mathbb{Z}}
\def\QQ{\mathbb{Q}}

\def\ideal{\triangleleft}
\def\lcm{\operatorname{LCM}}

\def\abs#1{\lvert#1\rvert}

\def\ntame{n_\text{tame}}
\def\nwild{n_\text{wild}}

\def\isom{\cong}

\def\uhat{\hat{u}}
\def\Vbar{\bar{V}}

\usepackage{color}
\definecolor{darkred}{rgb}{0.6, 0.0, 0.0}
\definecolor{darkgreen}{rgb}{0.0, 0.6, 0.0}
\def\editstart{\begingroup\marginpar{\color{black}EDIT}\color{darkred}}
\def\editswap{\color{darkgreen}}
\def\editend{\endgroup}
\long\def\editstart#1\editswap{}
\def\editend{}
\newcommand{\edit}[2]{\editstart#1\editswap#2\editend}

\begin{document}

\title{3-torsion and conductor of genus 2 curves}

\author{Tim Dokchitser}
\address{Department of Mathematics, University of Bristol, Bristol BS8 1TW, UK}
\email{tim.dokchitser@bristol.ac.uk}

\author{Christopher Doris}
\address{Department of Mathematics, University of Bristol, Bristol BS8 1TW, UK}
\email{cd14919@bristol.ac.uk}


\begin{abstract}
We give an algorithm to compute the conductor for curves of genus 2. It is based on the 
analysis of 3-torsion of the Jacobian for genus 2 curves over 2-adic fields.
\end{abstract}

\maketitle
%

\section{Introduction}
\label{sIntro}

One of the main arithmetic invariants of a curve $C/\Q$ (or over a number field) 
is its \emph{conductor}.
It is a representation-theoretic quantity measuring the arithmetic complexity of $C$,
and it is particularly important in the considerations that involve 
Galois representations or $L$-functions of curves. 

In practice, the conductor is difficult to compute.
It is defined as a product $N=\prod_p p^{n_p}$ over primes $p$, so the problem is 
computing the local \emph{conductor exponents} $n_p$; these are functions of $C/\Q_p$. 
For elliptic curves (genus 1), the problem of computing $n_p$ is solved
with Tate's algorithm \cite{TatA} and Ogg-Saito formula \cite{Ogg, Sai}.
In genus 2 and $p\ne 2$ there is an algorithm of Liu \cite{LiuA} via the 
Namikawa--Ueno classification \cite{NU}, and for hyperelliptic curves of arbitrary genus 
there is a formula for the conductor \cite{M2D2}, again for $p\ne 2$.

As the global conductor $N$ requires the knowledge of $n_p$ for \emph{all} primes~$p$, 
including $p=2$,
it is currently only provably computable for elliptic curves, and for quotients of 
modular curves using modular methods (see e.g. \cite{FLSSSW}). In practice, 
one can guess $N$ from the functional equation of the $L$-function 
(see e.g. \cite{computel, BSSVY}), but this approach is conditional
on the conjectural analytic continuation of the $L$-function, and is basically
restricted to reasonably small~$N$.

In this paper, we propose an (unconditional) algorithm to compute the conductor 
for curves of genus 2. The case to consider is $p=2$, 
so from now on $C$ will be a non-singular projective curve of genus 2, defined 
over a finite extension $K$ of $\Q_2$. Recall that the conductor exponent 
is the sum of the \emph{tame} and \emph{wild} parts (see \S\ref{sCond}),
$$
  n_2 = n = \ntame + \nwild.
$$
The difficult one is the wild part, which is the Swan conductor of the $l$-adic 
Tate module of the Jacobian $J/K$ of $C/K$, for any $l\ne 2$. 
We will take $l=3$ and use that $\nwild$ can be computed from the action of $\Gal(\bar K/K)$
on the 3-torsion $J[3]$. The equations defining $J[3]$ as a scheme are well-known
in genus 2 (see \S\ref{s3tors} or \cite{BFT}) and we use Grobner basis machinery to convert
them essentially to a univariate equation of degree $80 =|J[3]\setminus \{0\}|$. The problem 
then becomes to compute the Galois group of this polynomial, and enough information about
the inertia action on the roots to reconstruct the conductor. This is the core of the paper 
(\S\ref{sec/wild}). In particular, we discuss how to guarantee that the results are provably
correct (\S\ref{ssProvability}).

As for the tame part, it can be computed from the regular model of $C/K$, which is in principle 
accessible: take any model of $C$ over the ring of integers of $K$, and perform repeated
blowups until it becomes regular\footnote{Then $\ntame=4-2d_a-d_t$, where $d_a$ (`abelian part')
is the sum of genera of reduced components of the special fibre of the model, and 
$d_t$ (`toric part') is the number of loops}.
However, the algorithm to compute a regular model is currently only partially implemented
in Magma \cite{Magma}, and so we complement our algorithm with a result that determines 
$\ntame$ from elementary invariants, in the majority of the cases (Theorem \ref{thmtame}).

\edit{}{
An alternative approach to getting the conductor would be to find
a Galois extension $F/K$ where $C$ acquires semistable reduction and a semistable model over $F$, 
and analyse the action of inertia of $F/K$ on the model. From this one can determine 
the $l$-adic representation $V_l J$, in particular the conductor exponent;
see e.g. \cite[\S6]{hq}.
Moreover, that there are more compact polynomials defining such an $F$ in the case of genus 2, $p=2$
than the degree 80 3-torsion polynomial. For example, there is the monodromy polynomial of 
Lehr-Matignon in the potentially good reduction case, of degree 16 \cite[\S3]{LM}.
However, the splitting field of any such polynomial would have ramification degree no less 
that that of $K(J[3])/K$, by the Serre-Tate theorem \cite[Cor. 2]{ST}. 
So such a field (and the model of $C$ over it) would be still prohibitively 
large to compute, and our algorithm avoids this.
}

We end by noting that the core of the paper is a special test case of a general 
algorithm (in progress) to find Galois groups over local fields \cite{DPhD}.
Regarding Groebner bases,
the algorithm would be accelerated by an algorithm to solve multivariate systems of equations
$p$-adically (see Remark~\ref{accelgroebner}). This is also work in progress.
Finally, it should be possible to extend the algorithm to compute the conductor to function fields 
of characteristic~2 as well, by modifying the equations of the curve and its 3-torsion 
in~\S\ref{s3tors} appropriately.

This algorithm has been implemented as a Magma package \cite{3TorsionCode}, and has been used to verify most of the genus 2 curves in the LMFDB (\S\ref{sec/implementation}).

\subsection*{Acknowledgements}

This research is partially supported by an \hbox{EPSRC} grant EP/M016838/1 
`Arithmetic of hyperelliptic curves' and by GCHQ. 
We would like to thank David Roberts for helpful discussions \edit{}{and the referees for their suggestions}.

\newpage

\section{Notation}
\label{sCond}

\noindent
Throughout the paper, we use the following notation:

\smallskip

\begin{tabular}{llllll}
$K, L,...$ & local fields, of residue characteristic $p$\cr
$\cK, \cL,...$ & global fields\cr
$G_K$ & $=\Gal(\bar K/K)$, the absolute Galois group of $K$\cr
$I_K<G_K$ & its inertia group\cr
$T$ & $\Zl$-module with an action of $G_K$, with $l\ne p$\cr
$V$ & the associated $l$-adic representation $T\tensor_{\Zl}\Ql$\cr
$\bar V$ & the reduction $T\tensor_{\Zl}\F_l$\cr
$G^u$ & upper numbering of ramification groups\cr
$G_v$ & lower numbering of ramification groups\cr
$n = \ntame + \nwild$ & conductor exponent\cr
\end{tabular}

\smallskip

We are interested in the situation that $J/K$ is an abelian variety,
$T=T_l J$ is its $l$-adic Tate module, $V=V_l J$ and $\bar V=J[l]$ is its $l$-torsion.
Recall that the conductor exponent of such a representation is given by
(see e.g. \cite{Ulmer}) 
%
%
%
%
%
$$
  n(V) = \int_{-1}^\infty \codim V^{G_K^u} du, 
$$
with 
$$
  \ntame(V) = \int_{-1}^0 \qquad\text{and}\qquad \nwild(V) = \int_0^\infty.
$$
For $u>0$, $G_K^u$ is pro-$p$, and \cite[\S6]{Ulmer}
$$
  \codim V^{G_K^u} = \codim {\bar V}^{G_K^u}.
$$
Our approach is that we will compute $\ntame(V)$ as the codimension of inertia invariants $V^{I_K}$,
and the wild conductor exponent as 
$$
  \nwild(V) = \int_0^\infty \codim {J[l]}^{G_K^u} du,
$$
and replacing $G_K$ by $\Gal(K(J[l])/K)$.

\comment

\bigskip
{\bf old stuff}
\bigskip

Letting \(v_0=0\), \(v_1=1\) and \(v_2,\ldots,v_t\) be the distinct ramification breaks of \(G\), then
\[
n(\rho) = \sum_{i=0}^{t-1} \frac{1}{(G_0 : G_{v_i})} (v_{i+1} - v_i) (\dim V - \dim V^{G_{v_i}})
\]
and \(n_\text{tame}(\rho)\) and \(n_\text{wild}(\rho)\) are given by the \(i=0\) and \(i>0\) terms respectively.

Letting \((u_i)_{i=0}^t\) be the corresponding upper ramification breaks, so that \(G_{v_i} = G^{u_i}\), then the relationship
\[\frac{u_{i+1} - u_i}{v_{i+1} - v_i} = \frac{1}{(G_0 : G_{v_i})}\]
implies
\[n(\rho) = \sum_{i=0}^{t} (u_{i+1} - u_i)(\dim V - \dim V^{G^{u_i}}).\]
Therefore to compute \(n\) it suffices to know the upper ramification breaks \(u_i\) and the dimensions \(\dim V^{G^{u_i}}\).

In our case, we are given a hyperelliptic curve \(C : y^2 = f(x)\) of genus 2 defined over a number field \(K\) and we wish to compute its conductor exponent at a prime \(\fp\). That is, we wish to compute the conductor exponent at \(\fp\) of the Tate module \(T := T_\ell J\) where \(\ell\) is a rational prime coprime to \(\fp\) and \(J\) is the Jacobian of \(C\). Concretely, the Tate module is the limit of \(G\)-modules \(J[\ell^i] \isom (\ZZ/\ell^i\ZZ)^4 \setminus 0\) where \(G:=\Gal(\bar K/K)\) acts linearly giving a representation \(\rho : G \to \GL_4(\ZZ_\ell)\).

Composing with the residue map, we get the finite representation \(\bar\rho : G \to \GL_4(\FF_\ell)\) and we ask how \(\rho\) and \(\bar\rho\) are related. The answer is that they have the same wild conductor exponents, so we shall concern ourselves with computing this.

\endcomment

\section{Tame conductor exponent}
\label{sTame}

Let $K$ be any non-Archimedean local field, $J/K$ a $g$-dimensional abelian variety, 
and $l$ a prime different from the 
residue characteristic of $K$. Write $T=T_l J$ for the $l$-adic Tate module of $J/K$ and $V=V_l J=T_l J\tensor_{\Z_l}{\Q_l}$,
both viewed as representations of the inertia group $I_K<G_K$. 

Recall\edit{}{\footnote{
These are `standard' facts that we found a little hard to locate in the literature, 
but they are summarised in \cite{CFKS} \S 2.10: for the 
existence of a $\Gal(\bar K/K)$-stable filtration that forces the Galois group action to be 
upper-triangular see \cite[p.13, 2nd half]{CFKS}; for the fact that the representations on 
the graded pieces $\chi$ and $\rho$ are independent of $l$ see \cite[p.13, bottom]{CFKS}, and 
for the maps between them and the monodromy pairing \cite[pp. 12,14]{CFKS}. 
See also forthcoming paper \cite{skew}.}
}
that there is a canonical filtration on $T$  
coming from the toric part and the abelian part of $J$ over a field where it acquires 
semistable reduction. \edit{See e.g. \cite{CFKS} \S 2.10 for details.}{}
With respect to this filtration, $I_K$ acts on $T$ as
\begin{equation}\label{Iaction}
\begin{pmatrix}
 \chi & *    & N  \cr
 0    & \rho & *     \cr
 0    & 0    & \hat\chi  \cr
\end{pmatrix}
\end{equation}
with $\chi: I_K\to\GL_t(\Z_l)$, 
{$\rho: I_K\to\GL_{2a}(\Z_l)$}
continuous with finite image
(t=`toric', a=`abelian', 
{$2t+2a=\rk_{\Z_l}T=2g$}), and $\hat\chi$ the dual of $\chi$.
The `monodromy matrix' $N$ has $\Z$-coefficients, and
$\chi$ factors through $\GL_t(\Z)$ as well.
In particular, $\chi\tensor\Q_l$ is 
self-dual with determinant of order 1 or~2. Consequently, the same holds
for~$\rho\tensor\Q_l$, as $\det\eqref{Iaction}=1$ by the Weil pairing.

%
%

\makeatletter
\def\hetype#1{\hbox{$\tn@main#1\tn@end\tn@end$}}
\def\tn@und#1#2{{\tn@dound #2\tn@end\tn@end}\tn@main}
\def\tn@tri#1#2#3{{\smash{\raise2pt\hbox{$\scriptscriptstyle #3$}}}\tn@main}
\def\tn@main#1{%
   \ifx#1\tn@end\relax\else   
   \ifx#1__\expandafter\tn@und\else
   \ifx#1^ ^\expandafter\tn@tri\else     
   \ifx#1.\hbox{$\cdot$}\else
   \ifx#1:\hbox{:}\else
   \ifx#1o\hbox{$\hspace{0.6pt}\circ\hspace{0.6pt}$}\else
   \ifx#1x\hbox{$\hspace{0.6pt}\times\hspace{0.6pt}$}\else   
   \ifx#1I\hbox{I}\else   
   \ifx#1U\hbox{U}\else   
   #1\fi\fi\fi\fi\fi\fi
   \expandafter\expandafter\expandafter\expandafter
   \expandafter\expandafter\expandafter\tn@main\fi\fi\fi}
\def\tn@dound#1{%
  \ifx#1\tn@end\relax\else
  \ifx#1~\hbox{$\FrobL$}\else #1\fi\expandafter\tn@dound\fi}
\let\tn@end\relax
\makeatother

Now, we specialise to the case when $J=\Jac C$ is the Jacobian of a genus~2 curve and $l=3$.
We will explain in \S\ref{sec/wild} how to compute the image $I$ of $I_K$ in $\Aut J[3]$ and 
the dimension of inertia invariants $\dim J[3]^I$. 
{
We can also compute $t$ and $a$ using a theorem of Liu \cite[Thm 1]{Liu} that
determines the stable type of $C/K$ from the Igusa invariants of the curve. 
There are 7 possible stable types in genus 2, in other words possibilities for stable reduction. 
(For elliptic curves there are 2 types of stable reduction --- good and multiplicative.)
They are listed as cases I, II, ..., VII in Liu's theorem, and in the notation of \cite{Hyble} they are denoted
$
  \hetype{2},\>\> \hetype{1_n},\>\> \hetype{I_{n,m}},\>\> \hetype{U_{n,m,r}},\>\> \hetype{1x1},\>\> \hetype{1xI_n},\>\> \hetype{I_nxI_m}.
$
\edit{}{
The special fibres are as follows, with numbers above the components indicating geometric genus:

\tikzset{ 
  al/.style={above left=-0.02 and -0.1,scale=0.7},
  ar/.style={above right=-0.02 and -0.1,scale=0.7}
}

\begin{figure}[h]
\begin{center}
\begin{tikzpicture}[scale=0.35]
\draw[thick] (0,0) -- +(3,0) node[al] {2};
\node at (1.5,-1) {$2$};
\end{tikzpicture}
\qquad
\begin{tikzpicture}[scale=0.35]
\draw[thick] (0,0) -- +(1,0) arc (270:360:2 and 1) arc (0:180:1) arc (180:270:2 and 1) -- +(1,0) node[al] {1};
\node at (2,-1) {$1_n$};
\end{tikzpicture}
\qquad
\begin{tikzpicture}[scale=0.35]
\draw[thick] (0,0) -- +(1,0) arc (270:360:2 and 1) arc (0:180:1) arc (180:270:2 and 1) -- +(1,0) arc (270:360:2 and 1) arc (0:180:1) arc (180:270:2 and 1) -- +(1,0) node[al] {0};
\node at (3.5,-1) {$\mathrm I_{n,m}$};
\end{tikzpicture}
\qquad
\begin{tikzpicture}[scale=0.35]
\draw[thick] (0,0) -- (1,0) .. controls (2,0) and (2,2) .. (3,2) .. controls (4,2) and (4,0) .. (5,0) .. controls (6,0) and (6,2) .. (7,2) -- (8,2) node[al] {0};
\draw[thick] (0,2) -- (1,2) .. controls (2,2) and (2,0) .. (3,0) .. controls (4,0) and (4,2) .. (5,2) .. controls (6,2) and (6,0) .. (7,0) -- (8,0) node[al] {0};
\node at (4,-1) {$\mathrm U_{n,m,r}$};
\end{tikzpicture}
\par\vspace*{1em}
\begin{tikzpicture}[scale=0.35]
\draw[thick] (0,0) node[ar] {1} .. controls (2,0) .. (3,2);
\draw[thick] (2,2) .. controls (3,0) .. (5,0) node[al] {1};
\node at (2.5,-1) {$1 \times 1$};
\end{tikzpicture}
\qquad
\begin{tikzpicture}[scale=0.35]
\draw[thick] (0,0) node[ar] {1} .. controls (2,0) .. (3,2);
\draw[thick] (2,2) .. controls (3,0) .. (5,0) arc (270:360:2 and 1) arc (0:180:1) arc (180:270:2 and 1) -- +(1,0) node[al] {0};
\node at (4,-1) {$1 \times \mathrm I_n$};
\end{tikzpicture}
\qquad
\begin{tikzpicture}[scale=0.35]
\draw[thick] (0,0) node[ar] {0} -- +(1,0) arc (270:360:2 and 1) arc (0:180:1) arc (180:270:2 and 1) .. controls +(2,0) .. +(3,2);
\draw[thick] (5,2) .. controls +(1,-2) .. +(3,-2) arc (270:360:2 and 1) arc (0:180:1) arc (180:270:2 and 1) -- +(1,0) node[al] {0};
\node at (5.5,-1) {$\mathrm I_n \times \mathrm I_m$};
\end{tikzpicture}
\caption{The 7 stable reduction types for genus 2.}
\end{center}
\end{figure}
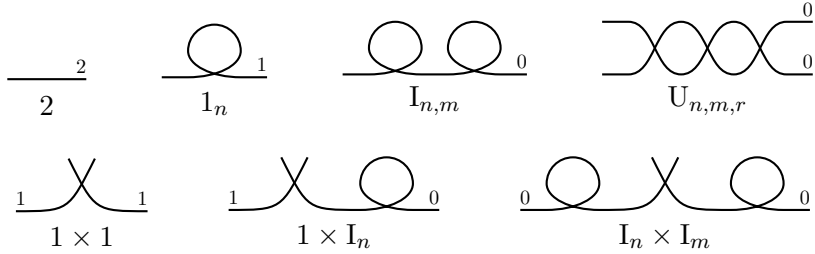
}

\par
\noindent
Of these, types \hetype{2} and \hetype{1x1} have $t=0, a=2$ (potentially good reduction of~$J$), types
\hetype{1_n} and \hetype{1xI_n} have $t=a=1$ (mixed), and \hetype{I_{n,m}}, \hetype{U_{n,m,r}} and
\hetype{I_nxI_m} have $t=2, a=0$ (potentially totally toric reduction).

The main result of this section recovers the tame conductor exponent of $J/K$ from the
invariants $I$, $\dim J[3]^I$ and $t$, when this is possible:

%

\begin{theorem}
\label{thmtame}
Let $K$ be a non-Archimedean local field of residue characteristic $\ne 3$ and $C/K$ a genus 2 curve with
Jacobian $J/K$. Write 

\begin{tabular}{llllll}
$I$ &=& image of inertia $I_K<G_K$ in $\Aut J[3]$ (so $I<\Sp_4(\F_3)$), \cr
$d$ &=& $\dim (V_3 J)^I$ (so $0\le d\le 4$),\cr
$\bar d$ &=& $\dim J[3]^I$ (so $0\le \bar d\le 4$),\cr
$t$ &=& potential toric dimension of $J$ (so $0\le t\le 2$),\cr
$f$ &=& $4-d$ $=\ntame(V_3 J)=\ntame(J/K)$ (so $0\le f\le 4$).\cr
\end{tabular}

\noindent
Then $\bar d\ge d$ and so $f\ge 4-\bar d$. Moreover,
\begin{enumerate}
\item 
If $\bar d=0$ then $f=4$.
\item
If $\bar d=4$ then $d=4-t$ and $f=t$.
\item
Suppose $J$ has potentially good reduction ($t\!=\!0$). If $|I|\!=\!3$ and $\bar d\!=\!2$ then $f\!=\!4$;
in all other cases, $f$ is the smallest even integer $\ge 4\!-\!\bar d$.
\item
If $(t,|I|) \in \{ (1,3),(2,3),(1,2),(1,6)\}$ then $f$ is not uniquely determined as a 
function of $t$, $I$ and $\bar d$.
\item
If $(t,|I|)=(2,9)$ then $f=4$; in all other cases not covered, $f=3$.
\end{enumerate}
\end{theorem}

\begin{proof}
Write $T=T_3 J$, $V=V_3 J$.
Note that after tensoring \eqref{Iaction} with $\Q_3$ and a suitable change of basis,
both $*$'s can be made 0 and $N$ a $t\times t$ identity matrix. In particular,
\begin{equation}\label{Iinv}
  V^{I_K} = \chi^I\oplus \rho^I, \qquad f = 4 - \dim\chi^I - \dim\rho^I.
\end{equation}
If $V$ has an $I_K$-invariant subspace of dimension
$d$, its intersection with $T$ gives a rank $d$ saturated sublattice of $T$,
whose reduction contributes at least dimension $d$ to $J[3]^I$. This shows that $\bar d\ge d$, 
and implies (1).

(2) 
{By Raynaud's semistability criterion \cite[Prop 4.7]{SGA71},
$J$ is semistable if $J[m]$ is unramified for some $m\ge 3$ coprime to the residue characteristic.}
Here $I_K$ acts trivially on $J[3]$, and so $J$ is semistable.
In other words, $f=t$ and $d=4-t$.

For the remainder of the proof, we assume $\bar d\in \{1,2,3\}$.

(3) By Serre-Tate's theorem \cite[Cor. 2]{ST}, $J$ has good reduction over $K(J[3])$;
that is, $I_K$ acts on $V_3 J$ through $I$. By Poincare duality, this representation has
even-dimensional inertia invariants, in other words $d$ is even. As $d\le \bar d\in \{1,2,3\}$,
the only possibility for {$f=4-d$} not to be the smallest even integer $\ge 4-\bar d$ is when
$d=0$ and $\bar d\in\{2,3\}$. Suppose we are in that case.

{
Consider the possibilities for $I<\Sp_4(\F_3)$. Note that $3$ divides $|I|$, for otherwise
the classical representation theory of $I$ agrees with its modular representation over $\F_3$, 
implying $d=\bar d$. Also note that $C_3\times C_3$ is not a quotient of $I$, as 
the residue characteristic of $K$ is not 3, and tame inertia is cyclic.
Computing in Magma \cite{Magma}, we find that that $\Sp_4(\F_3)$ has 162 conjugacy classes of subgroups,
of which 5 satify the three properties (a) order multiple of 3, (b) no $C_3\times C_3$-quotient, and 
(c) $\bar d\in\{2,3\}$. 
Call them $H_1, H_2, H_3\iso C_3$, $H_4\iso C_6$ and $H_5\iso\SL_2(\F_3)$. 

By the classification of integral $C_p$-lattices \cite{Die, Rei}, there are two indecomposable 
$\Z_3[C_3]$-lattices, up to isomorphism: the trivial lattice of rank 1, and a lattice $\Lambda$ of rank 2 on which 
the generator of $C_3$ acts as $\smallmatrix {-1}1{-1}0$; every finite rank $\Z_3[C_3]$-lattice 
is a direct sum of these. If $I\iso C_3$, then as $d=0$, we must have $T\iso \Lambda\oplus\Lambda$, and it has 
$\bar d=2$ as claimed.

It remains to show that $I\in\{H_4,H_5\}$ with $d=0$ is impossible. 
Suppose we are in this case, and let $z\in I$ be the unique central element of order~2.
As above, the classical representation theory of the group $\langle z\rangle\iso C_2$ agrees with its 
modular representation over $\F_3$.
In both $H_4$ and $H_5$ the action of $z$ on $\bar V=J[3]$ has two $+1$ 
and two $-1$ eigenvalues. The same is therefore true for $V$; moreover, $V=V^+\oplus V^-$ 
and $T=T^+\oplus T^-$ decompose into the two 2-dimensional eigenspaces for $z$ 
and this decomposition induces the one on $J[3]$.

The group $\SL_2(\F_3)$ has 3 one-dimensional complex representations
factoring through $\SL_2(\F_3)/Q_8\iso C_3$, three faithful 2-dimensional ones in which 
$z$ acts as $-1$, and a 3-dimensional one with $z$ acting as $+1$. 
Thus, when $I$ is $H_4$ and $H_5$, the space $T^+$ must be a representation of the unique 
$C_3$ quotient of $I$. It has no trivial subrepresentations (as $d=0$), so $T^+\iso\Lambda$
as a $\Z_3[C_3]$-module. 
But then
$$
  \bar d = \dim (\Lambda\tensor\F_3)^{C_3} + \dim (T^-\tensor\F_3)^{I} = 1 + 0,
$$
contradicting the assumption $\bar d\in\{2,3\}$.

%
%
%

}

(4) The following curves give examples over $\Q_2$ that prove that 
$f$ is not a function of $t$, $I$ and $\bar d$, as claimed.
(In each case, $f$ can be determined by computing the regular model.)
$$
\begin{array}{lll@{\hskip 3em}l@{\hskip 3em}l}
t & I & \bar d & f & C/\Q_2 \cr
\hline
1 & C_3 & 3 & 1 & y^2 = x^6+4x^4+2x^3+4x^2+1 \cr         
1 & C_3 & 3 & 3 & y^2 = 4x^6-20x^4-8x^3+21x^2+22x+13 \cr 
\hline
2 & C_3 & 2 & 2 & y^2 = x^6+6x^4-7x^2+16 \cr
2 & C_3 & 2 & 4 & y^2 = 5x^6+4x^3-12 \cr 
\hline
1 & C_2 & 2 & 2 & y^2 = -x^6+6x^4-x^2-8 \cr  
1 & C_2 & 2 & 3 & y^2 = x^6-6x^4+x^2+8 \cr 
\hline
1 & C_6 & 1 & 3 & y^2 = x^6-6x^4+5x^2+8 \cr  
1 & C_6 & 1 & 4 & y^2 = x^6-31x^4-25x^2-32 \cr 
\end{array}
$$

(5) 
To deal with all the remaining cases,
first suppose that $J$ has totally toric reduction over $K(J[3])$,
in other words $t=2$. In the notation of \eqref{Iaction}, we have a homomorphism
$$
  \chi: I \lar \GL_2(\Z) \quad(\injects \GL_2(\Z_3))
$$
whose image we denote by $\bar I$ and whose kernel is $C_1$ or $C_3$.
%
Finite subgroups of $\GL_2(\Z)$ are contained in $D_4$ or $D_6$. Of those,
$D_3$, $D_6$ only occur as inertia groups in residue
characteristic 3, and $C_2^2$, $C_4$, $C_6$, $D_4$ have an element acting as 
$-1$, forcing $\bar d=0$ (case (1)). The remaining possibilities are
$$
  \bar I\in\{C_1, C_2, C_3\}, \qquad I\in \{C_1,C_2,C_3,C_6,C_9\}.
$$
We have excluded $I=C_1$ (case (1)) and $I=C_3$ (case (4)).
When $I=C_9$, its image $\bar I\iso C_3$ has no invariants, and so $f=4$
(proving the case $(t,|I|)=(2,9)$). The only remaining case is $\bar I=C_2$,
acting with eigenvalues $+1,-1$ (otherwise $\bar d\in\{0,4\}$ again).
In this case, the full action on $T$ is of the form
$$
\smaller[2]
\begin{pmatrix}
1 & 0 & * & 0 \cr
0 & -1 & 0 & * \cr
0 & 0 & 1 & 0 \cr
0 & 0 & 0 & -1 \cr
\end{pmatrix}
$$
in some basis, with non-zero $*$'s. This has one-dimensional invariants,
and so $f=3$, as claimed.

Finally suppose $t=1$, so that $I_K$ acts on $T$ as 
$$
\smaller[2]
\begin{pmatrix}
\chi & * & * & \ne 0 \cr
0 & a & b & * \cr
0 & c & d & * \cr
0 & 0 & 0 & \chi \cr
\end{pmatrix}
$$
As before, write $\rho$ for the representation $I_K\to \smallmatrix abcd$.
Because $I$ is not one of the already excluded groups $C_1, C_2, C_3, C_6$,
the image of $I_K$ under $\bar\rho=\rho\mod 3$ is not 
$C_1$ or $C_2$. But any other subgroup of $\GL_2(\Z_3)$ of 
finite order is either $D_3$, which cannot be a local Galois group, or 
$\bar\rho(I)$ has no invariants on $\F_3^2$. Hence $\bar\rho^{I_K}=0$,
and $J[3]^{I_K}=\chi^{I_K}$ has either dimension 0 (case (1)) or dimension 1 with $f=3$,
as claimed.
\end{proof}

\section{Wild conductor exponent}
\label{sec/wild}

Recall that we wish to compute 
$$
  \nwild = \int_0^\infty \codim J[3]^{G^u} du
$$ 
where \(G=G_K\). Note, however, that \(G_K\) acts on \(J[3]\) through its finite quotient \(\Gal(K(J[3])/K)\) so we may equally well take \(G=\Gal(K(J[3])/K)\) or any quotient in between.


The integrand here is decreasing, non-negative, integral and left-constant, so if we denote by \(u_1=0,u_2,\ldots,u_t\) the jump points in the integrand, then we get \[\nwild = \sum_{i=2}^{t} (u_i - u_{i-1}) \codim J[3]^{G^{u_i}}.\]

Let \(Z \in J[3]\) be a 3-torsion point and let \(L=K(Z)\) be the extension it generates. Then \(Z\) is fixed by \(G^u\) if and only if \(L\) is fixed by \(G^u\). Since \(G^u \normal G\), this occurs if and only if any \(K\)-conjugate of \(Z\) is fixed by \(G^u\). If \(\uhat=\uhat(L/K)=\inf\{u : L \text{ fixed by } G^u\}\) denotes the highest upper ramification break of \(L/K\), then this occurs if and only if \(\uhat \leq u\).

Hence, if \(Z_1,\ldots,Z_m\) are representatives of {the} \(K\)-conjugacy classes of \(J[3]\), generating extensions \(L_i/K\) with highest upper ramification break \(\uhat_i\) then letting \(u_0=-1 < u_1=0 < \ldots < u_t\) be the sorted elements of \(\{-1,0,\uhat_1,\ldots,\uhat_m\}\) we deduce \[\nwild = \sum_{i=2}^{t} (u_i - u_{i-1}) \left(2g - \log_3 \sum_{j : \uhat_j \leq u_i} (L_j : K) \right)\] since \(2g=\dim V\) and \((L_j:K)\) is the number of \(K\)-conjugates of {\(Z_j\)}.

We proceed by finding the extensions \(L_i/K\) explicitly, from which we compute \(\nwild\) via this equation.

\subsection{Equation for 3-torsion of genus 2 curves}
\label{s3tors}

As before, let $C/K$ be a curve of genus 2, with Jacobian $J$.
The linear system for the canonical divisor on $C$ yields a standard model
$$
  C: y^2 =f(x), \qquad \deg f=5\text{ or }6.
$$
The following statement is well-known (see e.g. \cite{BFT} proof of Lemma 3); in fact, it works
over any field of characteristic $\ne 2,3$.

\begin{proposition}
\label{3torsprop}
Non-zero elements of $J[3]$ are in 1-1 correspondence with ways of expressing $f$ in the form
$$
  f = (z_4 x^3 + z_3 x^2 + z_2 x + z_1)^2 - z_7 (x^2 + z_6 x + z_5)^3, \qquad z_i \in \bar K,
  \eqno{(*)}
$$
and this correspondence preserves the action of $G_K$. 
\end{proposition}

Explicitly, suppose $D$ is a divisor on $C$,
$$
  D = (P_1) + (P_2) - (\infty_1)- (\infty_2), \qquad P_i=(X_i,Y_i)
$$
for which $3D$ is principal, say $3D=\div g$.
Then {$g\in\langle 1,x,x^2,x^3,y\rangle$}. After a (unique) re-scaling, say
$$
   g = y + b_3 x^3 + b_2 x^2 + b_1 x + b_0.
$$
The norm 
$$
\begin{array}{llllllllllll}
  \Norm_{K(C)/K(x)}(g) &=& (b_3 x^3 \!+\! b_2 x^2 \!+\! b_1 x \!+\! b_0 -y)(b_3 x^3 \!+\! b_2 x^2 \!+\! b_1 x \!+\! b_0 \!+\!y) \cr
    &=& (b_3 x^3 \!+\! b_2 x^2 \!+\! b_1 x \!+\! b_0)^2 - f.
\end{array}
$$
is a function on $\P^1$ whose divisor 
$3(X_1)+3(X_2)-6(\infty)$ is a cube, and so
$$
   (b_3 x^3 + b_2 x^2 + b_1 x + b_0)^2-f = c_2 (x^2+c_1x+c_0)^3,
$$
as stated. In this form,
$$
  X_{1,2} = \,\text{roots of }\>x^2\!+\!c_1x\!+\!c_0=0, \qquad Y_i = \!- b_3 X_i^3 \!-\! b_2 X_i^2 \!-\! b_1 X_i \!-\! b_0.
$$

We view ($*$) as giving a system of 7 equations in the 7 variables $z_i$.

\subsection{Finding the 3-torsion fields}

Our goal, then, is to find the (\(K\)-isomorphism classes of) fields \(L/K\) generated by the (\(K\)-conjugacy classes of) solutions \(Z\) to the system of equations ($*$).

{
A general tool used to solve systems of polynomial equations such as this is to compute a Groebner basis for the polynomial ideal generated by the polynomials. Generically, a reduced sorted minimal Groebner basis with respect to the lexicographic ordering on variables will be a finite sequence of polynomials such that the first is univariate, the second is a polynomial in two variables, and so on. Then to solve the system, we first find a root of the first polynomial; then we substitute this value into the second polynomial, yielding a polynomial in one variable, and we find a root of this; we repeat this procedure. In the end, this will produce a sequence of roots which together are a solution to the system.

For our system in particular, the 80 roots come in pairs of the form
$$
  (Z_1,Z_2,Z_3,Z_4,Z_5,Z_6,Z_7),\quad(-Z_1,-Z_2,-Z_3,-Z_4,Z_5,Z_6,Z_7),
$$
and so generically there are 40 distinct values for $Z_7$, for each of these there is a unique value for $Z_6$ and $Z_5$ and two distinct values for $Z_4$, and for each of these there is a unique value for $Z_3$, $Z_2$ and $Z_1$.

In this generic case, the Groebner basis described above will be a sequence of 7 polynomials \(B_1,\ldots,B_7 \in K[z_1,\ldots,z_7]\) such that \(B_i \in K[z_i,\ldots,z_7]\), \(\deg_{z_i} B_i = d_i\) where \(d=(1,1,1,2,1,1,40)\).

Following the above discussion on solving systems using Groebner bases, we first factorize \(B_7 \in K[z_7]\) (of degree 40), let \(g\) be one of its irreducible factors, let \(M/K\) be the extension it defines, and let \(Z_7 \in M\) be a root of \(g\). Substituting this into \(B_6 \in K[z_6,z_7]\) we get \(B_6(z_6,Z_7) \in M[z_6]\), which is linear, and let \(Z_6\) be its root. Similarly we let \(Z_5\) be the root of \(B_5(z_5,Z_6,Z_7) \in M[z_5]\). Next, \(B_4(z_4,Z_5,Z_6,Z_7) \in M[z_4]\) is quadratic, so we factorize it, let \(h\) be one of its factors, let \(L/M\) be the extension it defines, and let \(Z_4 \in L\) be a root of \(h\). Continuing, we find unique \(Z_3\), \(Z_2\) and \(Z_1\) which together produces a solution \(Z=(Z_1,\ldots,Z_7)\). Repeating this for all factors \(g\) and \(h\) we find all solutions \(Z\) of the system (up to conjugacy) and the extensions \(L/K\) which they define.

If we are not in this generic case, then the Groeber basis is not of this form and there is some coincidence in the coordinates of some solutions of the 7 equations. If we apply a random Mobius transformation \(x \mapsto \frac{ax+b}{cx+d}\) to the defining polynomial \(f(x)\) then the curve it defines is isomorphic to the original but the solutions \(Z\) have moved, probably to the generic case. In practice, a small number of Mobius transformations is ever necessary to put the solutions into the generic case.
}

\begin{remark}
An algorithm of this sort would work with any ordering on \(\{z_1,\ldots,z_7\}\). This ordering was chosen because it allows us to factor a degree-40 polynomial followed by a quadratic, which is somewhat faster than just factoring a degree-80 polynomial required for other orderings.
\end{remark}

\subsection{Provability}
\label{ssProvability}

In practice, however, computing a Groebner basis of this sort is difficult. 
Groebner basis algorithms require exact fields, so in practice we represent $K$ as a completion
of a number field $\cK$ at some place {$\p\mid2$}, and $f(x)\in\cK[x]$.

The best known algorithm over number fields (and indeed the only algorithm which appears to run in feasible time on our problem) computes the basis modulo many primes and finds the global basis via the Chinese remainder theorem. The problem here is that a priori we cannot determine the size of the coefficients, and so a heuristic is used to decide if we have used enough primes to get the answer. The result is that the algorithm does not yield provable results. Nevertheless, it is possible to prove the output of the previous algorithm as follows.


Assuming the Groebner basis algorithm was correct, then any $Z=(Z_1,\ldots$, $Z_7)$ should be a solution to the original system of 7 equations ($*$) over $K$. With the following version of Hensel's lemma, we can show that \(Z\) is indeed very close to a unique genuine solution, and we can say how close.

The following version of Hensel's lemma is standard (see e.g. \cite{Kuh} Thm.~23
with $t=\det J_f(b)$, $s=v f(b)$ and $v J^*_f(b)f(b)\ge s$).

\begin{theorem}[Hensel's lemma for multivariate systems]
Suppose \(K\) is a local field and \(F = (F_1,\ldots,F_m) \in \cO_K[z_1,\ldots,z_m]\) is a system of \(m\) equations in \(m\) variables over \(\cO_K\) and \(Z=(Z_1,\ldots,Z_m) \in \cO_K^m\). Let \(s = \min_i v_K(F_i(Z))\) and let \(t = v_K J(F)(Z)\) where \(J(F)\) denotes the Jacobian determinant of \(F\) (the determinant of the \(m \times m\) matrix whose \((i,j)\)th entry is \(\frac{\partial F_i}{\partial z_j}\)). If \(s > 2t\) then there is a unique \(Z' \in \cO_K^m\) such that \(F(Z')=0\) and \(\min_i v_K(Z'_i - Z_i) \ge s-t\).
\end{theorem}

Since evaluating resultants, Jacobians and polynomials are just basic arithmetic, these operations can be performed provably, and hence applying Hensel's lemma we prove that each \(Z\) is indeed close to a unique solution \(Z'\) of the system of equations. Furthermore, Hensel's lemma gives us a method to compute \(Z'\) to any prescribed precision. We expect that \(Z = Z'\) but we do not prove so.

It remains to check that these solutions \(Z'\) generate the fields \edit{\(M\)}{\(L\)} and that they are distinct up to \(K\)-conjugacy.

{
Recall that we have \(L/M/K\) with \(M = K(Z_7)\), \(g(x) \in K[x]\) the minimal polynomial for \(Z_7\), and \(L=M(Z_4)\), \(h(x) \in M[x]\) the minimal polynomial for \(Z_4\). We also have \(Z'_7,Z'_4 \in L\) and want to prove that \(L=K(Z'_7,Z'_4)\). Since we expect that \(Z'_7=Z_7\), then we expect \(Z'_7\) is closer to \(Z_7\) than any other root of \(g\), and so by Krasner's lemma we conclude that \(M = K(Z_7) \subset K(Z'_7)\). Another application of Krasner's lemma on \(h\) and \(Z'_4\) implies that \(L = M(Z_4) \subset M(Z_4')\). Combining these, we deduce \(L = M(Z_4) \subset M(Z_4') \subset K(Z'_4,Z'_7) \subset L\) and hence \(L = K(Z'_4,Z'_7) = K(Z')\).

To check Krasner's lemma on a polynomial \(h \in K[x]\) and some \(Z \in \bar K\), note that it is equivalent to check that there is a root of \(h(x+Z)\) of higer valuation than all others. It is well-known that the Newton polygon of a polynomial measures the valuations of its roots, and therefore Krasner's lemma is applicable if and only if the Newton polygon of \(h(x+Z)\) has a vertex with abscissa 1. This condition is explicitly checkable.
}

Finally, if \(Z_7\) is a root of a factor \(g\) of \(B_7\) and \(Y_7\) is a root of a different factor of \(B_7\), then \(g(Z_7) = 0 \neq g(Y_7)\), so if we check that \(v(g(Z'_7)) > v(g(Y'_7))\) then we have proven that \(Z'_7 \neq Y'_7\). Performing a similar check on pairs of \(Z'_4\) determines that they are different. Together, this will prove that each pair of solutions is distinct.

By performing all these checks with large enough precision, we can determine whether or not the \(Z\) are a genuine set of distinct solutions generating the right fields. If any of these checks fails, then the Groebner basis algorithm was incorrect, and we should try the algorithm again with a lower heuristic chance of failure.

\begin{remark}
\label{rmk-global-proof}
There is a conceptually simpler method for provability. Letting \(I \ideal \cK[z_1,\ldots,z_7]\) be the ideal generated by the original system (\(*\)), and letting \(J\) be the ideal generated by the Groebner basis, then we wish to prove that \(I = J\). Since \(J\) is generated by a Groebner basis, there is a normal form for reduction modulo \(J\) and hence we can check that each generator of \(I\) is zero mod \(J\) and so deduce \(I \ideal J\). Additionally we know a priori that \(I\) has precisely 80 solutions, and from the structure of the Groebner basis that \(J\) has precisely 80 solutions. Combined, this implies \(I=J\).

We call this the \emph{global proof method} to distinguish it from the \emph{local proof method} above. In practice, unless the coefficients of \(f(x)\) are very small, the global method takes much longer than the local method. Over \(\QQ\), with small coefficients, the global method is typically around twice as quick, but this benefit quickly diminishes as the field degree increases.
\end{remark}

\subsection{Tame conductor exponent revisited}

In order to compute the tame conductor exponent using Theorem \ref{thmtame}, we require \(\bar{d} = \dim J[3]^{G_0}\) and \(|I|\). In previous sections we have already seen an algorithm to compute \(\dim J[3]^{G^u}\) for any \(u\) having already computed \(L_j/K\), so this is easy as a side-effect of previous work.

For \(|I|\), consider \(e = \lcm_j e(L_j/K)\), which again is easy to compute from \(L_j/K\). Clearly it is a divisor of \(|I|\). The following lemma shows that \(e\) is a good enough guess at \(|I|\) in the sense that {the statement of Theorem \ref{thmtame} depends only on \(t\), \(e\) and \(\bar d\)}.

\begin{lemma}
{
Let \(S=\{1,2,3,4,5,6,9,10,12,18\}\). If $e\in S$ or $\abs I \in S$ then $\abs I = e$. If \(e=80\) then \(\abs I = 160\). If \(e \in \{8,24\}\) then \(\abs I \in \{8,24\}\). Otherwise \(e \in \{16,32,48,64\}\) and \(\abs I \in \{16,32,48,64,96,128,192,384\}\).
}
\end{lemma}

\begin{proof}
Properties of the Weil pairing imply that \(I < \Sp_4(\FF_3)\). Letting \(W\) be the 2-Sylow subgroup of \(I\), ramification theory implies {\(W \normal I\)} and \(I/W\) cyclic. The lemma is proven by checking all groups \(I\) consistent with these facts.
\end{proof}

\section{The algorithm}

{
We use the following algorithm to compute the highest upper ramification break \(\uhat(L/K)\). It takes as input the extension \(L/K\) and returns the sequence \((u_i,v_i,s_i)_{i=0}^t\) where \(v_0=-1<v_1<\ldots<v_t\) are the breaks in the ramification filtration of \(L/K\) in the lower numbering, \(u_i\) are the corresponding breaks in the upper numbering, and \(s_i = \abs{\Gamma_{v_i}}\) are the sizes of the corresponding ramification subsets of the Galois set \(\Gamma\) of \(K\)-embeddings \(L \to \bar K\). In particular, \(\uhat(L/K) = u_t\).

See e.g. \cite[\S4--5]{GP} or \cite[\S3]{PS} for the definition of the ramification polynomial (the coefficients of which have valuation \(r_i\) in the algorithm), the ramification polygon \(P\) and its connection to the ramification filtration of \(L/K\). See e.g. \cite{Helou} for the connection of this filtration to the upper and lower ramification breaks and the Galois set \(\Gamma\).

\newpage

\begin{algorithmic}[1]
\def\Return{\State\textbf{return} }
\def\Assert{\State\textbf{assert} }
\def\LineComment#1{\vspace{1.5ex}\State \textit{(#1)}}
\LineComment{Compute the ramification polygon of \(L/U\)}
\State \(U \leftarrow\) the maximal unramified subextension of \(L/K\)
\State \(e \leftarrow (L:U)\)
\State \(E \leftarrow\) a defining Eisenstein polynomial for \(L/U\)
\State \(r_i \leftarrow \min_{j=i}^{e-1} v(E_j {j \choose i}) + \tfrac{j}{e}\) for \(i = 1,\ldots,e\)
\State \(P \leftarrow\) the lower convex hull of the points \((i,r_i)\) for \(1 \leq i \leq e\)

\LineComment{Compute \(u_i\), \(v_i\) and \(s_i=|\Gamma_{v_i}|\)}
\State \(u_0 \leftarrow -1\)
\State \(v_0 \leftarrow -1\)
\State \(s_0 \leftarrow (L:K)\)
\State \(t \leftarrow\) the number of faces of \(P\)
\ForAll{\(i=1,\ldots,t\)}
\State \(F \leftarrow\) the \(i\)th face of \(P\) from the right
\State \(v_i \leftarrow\) the negative of the gradient of \(F\)
\State \(s_i \leftarrow\) the abscissa of the right hand vertex of \(F\)
\State \(u_i \leftarrow u_{i-1} + \frac{s_{i-1}}{s_0}(v_i-v_{i-1})\)
\EndFor

\Return \(((u_i,v_i,s_i))_{i=0}^t\)
\end{algorithmic}
}

{Now} we present the final algorithm, which takes a polynomial \(f(x)\in\cK[x]\) of degree 5 or 6 over a number field \(\cK\) defining a hyperelliptic curve \(y^2=f(x)\), and a prime ideal \(\fp\) of \(\cK\) dividing 2, and returns the conductor exponent \(n_\fp\) of the curve at \(\fp\).

\begin{algorithmic}[1]
\def\Return{\State\textbf{return} }
\def\Assert{\State\textbf{assert} }
\def\LineComment#1{\vspace{1.5ex}\State \textit{(#1)}}
\LineComment{Apply Moebius transformations to \(f(x)\) until its 3-torsion points are in general position}
\Repeat
  \State choose {\(a,b,c,d \in \ZZ\)} so that \(ad-bc \neq 0\)
  \State \(\tilde{f} \leftarrow f(\frac{ax+b}{cx+d})(cx+d)^6\)
  \State \(F=(F_i)_{i=1}^7 \leftarrow\) coefficients of \[(z_1 + z_2 x + z_3 x^2 + z_4 x^3)^2 + z_7 (z_5 + z_6 x + x^2)^3 - \tilde{f}(x)\]
  \State \(B=(B_i)_i \leftarrow\) Groebner basis of \(F\)
\Until{\(B\) is in {generic} form} \label{alg-groebner-end}
\LineComment{Find the fields defined by each \(Z_7\)}
\State \(K \leftarrow \cK_\fp\)
\State \(S \leftarrow \) empty sequence
\State \(C \leftarrow \) empty sequence
\State \((g_i)_i \leftarrow\) irreducible factorization of \(B_7(x)\) over \(K\)
\ForAll{\(g_i\)}
  \State \(M \leftarrow\) the extension of \(K\) defined by \(g_i\)
  \State \(Z_7 \leftarrow\) a root of \(g_i\) in \(M\)
  \State \(Z_6 \leftarrow\) the root of linear \(B_6(x,Z_7)\) over \(M\)
  \State \(Z_5 \leftarrow\) the root of linear \(B_5(x,Z_6,Z_7)\) over \(M\)
  \LineComment{Find the fields defined by each \(Z_4\)}
  \State \((h_i)_i \leftarrow\) irreducible factorization of \(B_4(x,Z_5,Z_6,Z_7)\) over \(M\)
  \ForAll{\(h_i\)}
    \State \(L \leftarrow\) the extension of \(M\) defined by \(h_i\)
    \State \(Z_4 \leftarrow\) a root of \(h_i\) in \(L\) \label{alg-locproof-start}
    \State \(Z_3 \leftarrow\) the root of linear \(B_3(x,Z_4,Z_5,Z_6,Z_7)\) over \(L\)
    \State \(Z_2 \leftarrow\) the root of linear \(B_2(x,Z_3,Z_4,Z_5,Z_6,Z_7)\) over \(L\)
    \State \(Z_1 \leftarrow\) the root of linear \(B_1(x,Z_2,Z_3,Z_4,Z_5,Z_6,Z_7)\) over \(L\)
    {
      \LineComment{Check the solutions are valid with Hensel's lemma}
      \Assert \(Z\) is Hensel liftable to a solution of \(F\)
      \State \(Z' \leftarrow\) the Hensel-lifted solution (we expect \(Z'=Z\))
      \LineComment{Check the solutions generate the right fields with Krasner's lemma}
      \Assert the Newton polygon of \(g_i(x+Z'_7)\) has a vertex above 1
      \Assert the Newton polygon of \(h_j(x+Z'_4)\) has a vertex above 1
    }
    \LineComment{Check the solutions are distinct}
    \For{\((Y'_7, Y'_4) \in C\)}
      \Assert \(v_K(g_i(Z'_7)) > v_K(g_i(Y'_7))\) or \(v_K(h_i(Z'_4)) > v_K(h_i(Y'_4))\)
    \EndFor
    \State append \((Z'_7, Z'_4)\) to \(C\) \label{alg-locproof-end}
    \LineComment{Save \(L\)}
    \State append \(L\) to \(S\)
  \EndFor
\EndFor
\LineComment{Compute the tame and wild exponents from \(S\)}
\State \(\bar{d} \leftarrow\) the function \(u \mapsto \log_3 (1 + \sum_{L \in S : \uhat(L/K) \leq u} (L:K))\) \((= \dim \Vbar^{G^u})\)
\State \(e \leftarrow \lcm_{L \in S} e(L/K)\)
\State \(t \leftarrow\) potential toric dimension of \(J\)
\If{\(\bar{d}(0) = 0\)}
  \State \(\ntame \leftarrow 4\)
\ElsIf{\(\bar{d}(0) = 4\)}
  \State \(\ntame \leftarrow t\)
\ElsIf{\(t = 0\)}
  \If{\(e=3\) and \(\bar{d}(0) = 2\)}
    \State \(\ntame \leftarrow 4\)
  \Else
    \State \(\ntame \leftarrow\) smallest even integer \(\geq 4-\bar{d}(0)\)
  \EndIf
\ElsIf{\((t,e) \in \{(1,3),(2,3),(1,2),(1,6)\}\)}
  \State \(\ntame \leftarrow\) the tame exponent, computed from a regular model
\ElsIf{\((t,e)=(2,9)\)}
  \State \(\ntame \leftarrow 4\)
\Else
  \State \(\ntame \leftarrow 3\)
\EndIf
\State {\(u_0,\ldots,u_t \leftarrow\)} the sorted elements of \(\{\uhat(L/K) : L \in S\} \cup \{-1,0\}\)
\State \(\nwild \leftarrow \sum_{i=2}^{t} (u_i - u_{i-1})(4 - \bar{d}(u_i))\)
\State \textbf{return} \(\nwild + \ntame\)
\end{algorithmic}

\begin{remark}
\label{accelgroebner}
Note that the approach to solving the system of 7 equations in 7 variables is to compute a Groebner basis globally, and then solve this system locally. This is the only global aspect of the algorithm, and becomes the bottleneck when the global coefficients become large. An alternative approach is to solve the system of equations directly locally, perhaps using a Montes-type algorithm similar to univariate factorization algorithms which split the system into several smaller systems. This is the subject of ongoing research {\cite{DPhD}}.
\end{remark}

\begin{remark}
Recalling Remark \ref{rmk-global-proof}, if we wish to use the global proof method instead, then we can skip over lines \ref{alg-locproof-start}--\ref{alg-locproof-end} and instead insert after line \ref{alg-groebner-end} a check that each element of \(F\) reduces to \(0\) modulo \(B\).
\end{remark}

\section{Implementation}
\label{sec/implementation}

The algorithms described in this paper have been implemented {\cite{3TorsionCode}} in the Magma computer algebra system \cite{Magma} using a customized implementation of $p$-adics which removes most $p$-adic precision considerations from the user \cite{DExactPadics}. The implementation, modulo bugs, produces provable results at every step.

The LMFDB \cite{LMFDB} contains the 66,158 genus 2 hyperelliptic curves defined over \(\QQ\) computed by Booker et al \cite{BSSVY}. Of these, all but 1113 have discriminant of 2-valuation less than 12 and therefore their conductor exponent at 2 is computable via Ogg's formula. Our algorithm has been run on the 1113 remaining curves, using the global proof method (see Remark \ref{rmk-global-proof}). The computation took 9.4 core-hours in total on a 2.7GHz Intel Xeon, averaging 30 core-seconds per curve.

For all but 6 of these curves, the fast tame conductor algorithm of \S\ref{sTame} succeeds, and so we compute an entire conductor exponent at 2. For 4 of the remaining 6 curves, a regular model was quickly \edit{computable}{computed by Magma} (taking at most 10 seconds) and therefore the tame exponent was deduced this way. \edit{}{For the remaining 2 curves (labelled \texttt{3616.b.\discretionary{}{}{}462848.1} and \texttt{18816.d.\discretionary{}{}{}602112.1} in the LMFDB) a regular model was computed by hand.} In all of these cases, the exponent agrees with the unproven results of \cite{BSSVY} and therefore we have proven the conductors for all \edit{but 2 curves}{curves in the LMFDB}.

\edit{For the remaining 2 curves (labelled \texttt{3616.b.\discretionary{}{}{}462848.1} and \texttt{18816.d.\discretionary{}{}{}602112.1} in the LMFDB) a regular model was not found and so all we know is that \(n - \nwild \in \{1,2,3\}\) for these curves. These are consistent with the unproven conductor exponents.}{}

The run-time of the algorithm is usually dominated by the factorization of the degree-40 polynomial over \(K\), at least when the defining polynomial \(f(x)\) has fairly small coefficients. When these coefficients grow, the (global) Groebner basis algorithm dominates the run-time.

This gives some impetus towards developing a fully local algorithm as suggested in Remark \ref{accelgroebner}, since this will be independent of global coefficient sizes.

%
%

The implementation has also been tested on some curves defined over quadratic number fields. These results were confirmed by Schembri \cite{Schembri} by finding a corresponding Bianchi modular form whose level squared equals the conductor and proving the expected relationship between their $L$-functions using Faltings-Serre.

The run-time does not appear to grow much with the conductor exponent, as evidenced by the following graph summarizing the run-times of the algorithm on the LMFDB curves. It plots the mean run-time for each conductor exponent with a thick black line, and plots the 20-percentiles with thin gray lines.

\begin{center}
\begin{tikzpicture}
\pgfplotstableread{
 n num  mean   min   p10   p20   p30   p40   p50   p60   p70   p80   p90    max   p25   p75
 2   1 34.80 34.80 34.80 34.80 34.80 34.80 34.80 34.80 34.80 34.80 34.80  34.80 34.80 34.80
 3   1 31.34 31.34 31.34 31.34 31.34 31.34 31.34 31.34 31.34 31.34 31.34  31.34 31.34 31.34
 4  13 26.42 12.86 19.30 21.85 23.27 24.39 28.17 29.03 30.09 30.87 31.88  39.43 23.68 31.20
 5  42 21.73  4.23  5.75  6.33  6.78 14.40 21.02 23.49 28.32 31.20 33.30 131.59  6.41 29.38
 6  51 30.12 17.27 20.65 21.73 22.92 23.85 24.74 25.69 26.58 28.56 34.93 183.86 22.19 27.67
 7  94 22.25  4.86  6.82  9.07 18.86 21.75 22.44 24.79 25.92 27.41 31.76  85.48 10.42 26.45
 8 110 19.78  4.95  5.82  8.21 13.97 18.46 21.56 22.91 25.14 28.30 31.54  41.68 11.25 26.18
 9  61 28.41 20.96 23.32 24.13 25.51 25.88 27.00 27.68 29.14 33.55 38.18  50.78 24.55 30.23
10 108 37.09 15.37 20.97 23.50 24.75 25.61 26.55 27.51 29.00 31.31 82.86 155.12 24.13 30.10
11 150 31.08  5.47  8.56 10.12 17.56 22.59 24.62 26.33 27.61 29.85 84.37 152.15 14.34 28.26
12 247 35.34  6.00 11.22 20.98 23.65 25.65 27.97 29.83 32.44 34.37 42.66 246.02 22.63 33.24
13 105 35.75  8.76 12.69 17.62 20.28 25.63 29.23 31.13 33.18 52.60 84.03  99.65 19.31 35.81
14  80 30.63  8.61 12.90 22.02 24.51 26.52 28.88 31.14 32.97 39.75 45.98  82.54 23.20 37.15
15  18 31.23 17.21 20.52 21.83 24.21 24.22 26.78 32.86 34.40 38.71 46.70  66.45 23.65 37.24
16  17 21.52 17.95 18.55 20.20 20.32 20.80 21.61 21.78 22.26 22.33 25.14  26.67 20.49 22.53
17   6 30.66 25.19 25.74 27.32 26.81 32.16 29.74 34.36 33.26 38.67 36.51  38.67 27.06 32.71
18   5 26.80 18.94 23.33 20.40 28.66 27.45 29.74 29.55 29.36 29.79 29.76  29.80 29.26 29.80
19   4 25.16 20.98 21.43 21.24 21.04 27.16 25.09 23.01 29.39 29.11 28.83  29.48 21.14 29.25
}\data
\begin{axis}[ymin=0, ymax=60, xlabel=conductor exponent \(n\), ylabel=time (core-seconds)]
\addplot[thick] table [y=mean] {\data};
\addplot[gray] table [y=min] {\data};
\addplot[gray] table [y=p20] {\data};
\addplot[gray] table [y=p40] {\data};
\addplot[gray] table [y=p60] {\data};
\addplot[gray] table [y=p80] {\data};
\addplot[gray] table [y=max] {\data};
\end{axis}
\end{tikzpicture}
\end{center}

\end{document}